\documentclass[reqno,a4paper,10pt]{amsart}
\title[How roundoff errors help to compute the rotation set]{How roundoff errors help to compute the rotation set of torus homeomorphisms}
\date{\today}
\usepackage[latin1]{inputenc}
\usepackage[english]{babel}
\usepackage[T1]{fontenc}
\usepackage{amsfonts}
\usepackage{amsmath}
\usepackage{amssymb}
\usepackage{stmaryrd}
\usepackage{amsthm}
\usepackage{enumerate}
\usepackage{pgf,tikz}
\usetikzlibrary{arrows}
\usepackage[pdftex,colorlinks=true,linkcolor=blue,citecolor=blue,urlcolor=blue]{hyperref}
\usepackage{float}
\usepackage{changepage}
\usepackage{caption}

\author{Pierre-Antoine Guihéneuf}
\newtheorem{lemme}{Lemma}
\newtheorem{theoreme}[lemme]{Theorem}
\newtheorem{prop}[lemme]{Proposition}
\newtheorem{coro}[lemme]{Corollary}

\newtheorem*{theorem}{Theorem}

\theoremstyle{definition}
\newtheorem{definition}[lemme]{Definition}

\theoremstyle{remark}
\newtheorem{rem}[lemme]{Remark}

\newtheorem{ex}[lemme]{Example}

\newcommand{\N}{\mathbf{N}}
\newcommand{\Pb}{\mathcal{P}}
\newcommand{\R}{\mathbf{R}}

\newcommand{\T}{\mathbf{T}}

\newcommand{\Q}{\mathbf{Q}}

\newcommand{\varep}{\varepsilon}
\newcommand{\Hom}{\mathrm{Homeo}}

\newcommand{\Leb}{\mathrm{Leb}}
\newcommand{\ud}{\mathrm{d}}

\newcommand{\card}{\operatorname{Card}}

\newcommand{\Obs}{\operatorname{Obs}}
\newcommand{\diam}{\operatorname{diam}}
\newcommand{\dist}{\operatorname{dist}}

\begin{document}

\begin{abstract}
The goals of this paper are to obtain theoretical models of what happens when a computer calculates the rotation set of a homeomorphism, and to find a good algorithm to perform simulations of this rotation set. To do that we introduce the notion of observable rotation set, which takes into account the fact that we can only detect phenomenon appearing on positive Lebesgue measure sets; we also define the asymptotic discretized rotation set which in addition takes into account the fact that the computer calculates with a finite number of digits

It appears that both theoretical results and simulations suggest that the asymptotic discretized rotation set is a much better approximation of the rotation set than the observable rotation set, in other words we need to do coarse roundoff errors to obtain numerically the rotation set.
\end{abstract}

\email{pierre-antoine.guiheneuf@math.u-psud.fr}
\address{Laboratoire de mathématiques CNRS UMR 8628\\Université Paris-Sud 11, Bât. 425\\91405 Orsay Cedex FRANCE}
\subjclass[2010]{37M05, 37M25, 37E45, 37A05}
\keywords{Rotation set, computation, generic homeomorphism}

\maketitle

\section{Introduction}

In 1885 in \cite{Poincare}, H. Poincaré introduced one of the first dynamical invariants, the rotation number for orientation preserving homeomorphisms of the circle. This definition led him to the theorem of classification: every orientation-preserving homeomorphism of the circle of irrational rotation number
$\alpha$ is semi-conjugate to the rigid rotation of angle $\alpha$. Ever since, considerations about the rotation number have helped us to improve our knowledge about the dynamical behaviour of circle homeomorphisms (see for example \cite{MR538680}).

About a century after was introduced a generalisation to dimension 2 of this rotation number, the rotation set for homeomorphisms of the torus. Due to the loss of natural order on the phase space, there is no longer a single speed of rotation for orbits; informally the rotation set is then the set of all possible rotation speeds of all possible orbits. Like in dimension 1 this topological invariant gives precious informations about the dynamics of the homeomorphism. For example, depending of the shape of this set we can ensure the existence of periodic points (\cite{MR967632}, \cite{MR958891}). Moreover, the size of this convex set gives lower bounds on the topological entropy of the homeomorphism (\cite{MR1101087}, \cite{MR1213082} for an explicit estimation)\dots

The aim of this paper is to tackle the question of numerical approximation of the rotation set: given a homeomorphism of the torus homotopic to the identity, is it possible to compute numerically its rotation set? In particular, is it possible to detect its dimension? Is it possible to approximate it in Hausdorff topology?
\medskip

First of all, we build a theoretical model of what happens when we try to calculate the rotation set of an homeomorphism with a computer. To do that, we first have to take into account the fact that the computer can calculate only a finite number of orbits; in particular it will detect only phenomenon that occur on positive Lebesgue measure sets. This leads to the notion of \emph{observable rotation set}: a rotation vector is called observable if it is the rotation vector of an observable measure in the sense given by E. Catsigeras and H. Enrich in \cite{MR2852870}; more precisely a measure $\mu$ is observable if for every $\varep>0$, the set of points which have a Birkhoff limit whose distance to $\mu$ is smaller than $\varep$ has positive Lebesgue measure (see Definition \ref{DefObs}).

But this notion of observable measure does not take into account the fact that the computer uses finite precision numbers and can calculate only finite length orbits; it leads to the definition of the \emph{asymptotic discretized rotation set} in the following way. We fix a sequence of finite grids on the torus with precision going to 0; the discretized rotation set on one of these grids is the collection of rotation vectors of periodic orbits of the discretization of the homeomorphism on this grid (see Section \ref{DiscRot}); the asymptotic discretized rotation set is then the upper limit of these discretized rotation sets on the grids.

We focused mainly on the generic behaviour\footnote{A property will be called \emph{generic} if it is true on at least a countable intersection of dense open sets.} of both observable and asymptotic discretized rotation sets. We recall that a theorem of A. Passeggi \cite{rata} states that for a generic dissipative homeomorphism of the torus the rotation set is a polygon with rational vertices, possibly degenerated\footnote{Namely it can be a segment or a singleton.}\footnote{However there are open sets of homeomorphisms where the rotation set has nonempty interior.}. In this paper we will prove the following:

\begin{theorem}
For a generic dissipative homeomorphism, the observable rotation set is the closure of the set of rotation vectors corresponding to Lyapunov stable periodic points (Lemma \ref{ExGeneDissip}). Then, the convex hull of the observable rotation set, the convex hull of the asymptotic discretized rotation set and the rotation set are equal. Moreover if the rotation set has nonempty interior there is no need to take convex hulls, \emph{i.e.} the observable rotation set coincides with the asymptotic discretized rotation set and the rotation set (Propositions \ref{GeneDissip} and \ref{RotDiscrDissip}).
\end{theorem}

Thus, it is possible to obtain the rotation set of a generic dissipative homeomorphism from the observable or the asymptotic discretized rotation set. In other words, from the theoretical point of view, it is possible to recover numerically the rotation set of a generic homeomorphism. We now present the results in the conservative setting.

\begin{theorem}
For a generic conservative homeomorphism (\emph{i.e.} which preserves Lebesgue measure), we give a proof that the rotation set has nonempty interior (Proposition \ref{RotGeneCons}). Then, for a generic conservative homeomorphism, the observable rotation set consists in a single point: the mean rotation vector (Proposition \ref{ExGeneCons}). On the other hand the asymptotic discretized rotation set coincides with the rotation set (Corollary \ref{CoroRotDiscrCons}).
\end{theorem}

These results suggest the quite surprising moral that to recover the rotation set of a conservative homeomorphisms it is better to do coarse roundoff errors at each iteration. More precisely, if we compute a finite number of orbits with arbitrarily good precision and long length we will find only the mean rotation vector of the homeomorphism, but if we make roundoff errors while computing we will be able to retrieve the whole rotation set.
\medskip

We have performed numerical simulations to see whether these behaviours can be observed in practice or not. To obtain numerically an approximation of the observable rotation set, we have calculated rotation vectors of long segments of orbits for a lot of starting points, these points being chosen randomly fore some simulations and being all the points of a grid on the torus for other simulations. For the numerical approximation of the asymptotic discretized rotation set we have chosen a fine enough grid on the torus and have calculated the rotation vectors of periodic orbits of the discretization of the homeomorphism on this grid.

We have chosen to make these simulations on an example where the rotation set is known to be the square $[0,1]^2$. It makes us sure of the shape of the rotation set we should obtain numerically, however it limits a bit the ``genericity'' of the examples we can produce.

In the dissipative case we made attractive the periodic points which realize the vertex of the rotation set $[0,1]^2$. It is obvious that these rotation vectors, which are realized by attractive periodic points with basin of attraction of reasonable size, will be detected by the simulations of both observable and asymptotic discretized rotation sets; that is we observe in practice: we can recover quickly the rotation set in both cases.

In the conservative setting we observe more or less the surprising behaviour predicted by the theory: when we compute the rotation vectors of long segments of orbits we obtain mainly rotation vectors which are quite close to the mean rotation vector, in particular we do not recover the initial rotation set. More precisely, when we perform simulations with less than one hour of calculus we only obtain rotation vectors close to the mean rotation vector, and when we let three hours to the computer we only recover one vertex of the rotation set $[0,1]^2$. On the other hand, the rotation set is detected very quickly by the convex hulls of discretized rotation sets (less than one second of calculus). Moreover, when we calculate the union of the discretized rotation sets over several grids to obtain a simulation of the asymptotic discretized rotation set, we obtain a set which is quite close to $[0,1]^2$ for Hausdorff distance. As for theoretical results, this suggests the following lesson:

\emph{When we compute segments of orbits with very good precision it is very difficult to recover the rotation set. However, when we decrease the number of digits used in computations we can obtain quickly a very good approximation of the rotation set. In fact, we have to adapt the precision of the calculus to the number of orbits we can obtain numerically.}

This phenomenon can be explained by the fact that each grid of the torus is stabilized by the corresponding discretization of the homeomorphism. Thus, there exists an infinite number of grids such that every periodic point of the homeomorphism is shadowed by some periodic orbits of the discretizations on these grids.
\medskip

\begin{rem}
In this paper we do not study the approximation of homeomorphisms by multivalued maps (see for instance \cite{MR2776399}); this point of view can maybe detect better the rotation set for examples that we do not have studied here, in particular when the vertices of the rotation set are obtained as hyperbolic periodic points of a diffeomorphism.
\end{rem}

\section{Notations and preliminaries}

\subsection{Notations}

The set of homeomorphisms of $\T^2$ will be denoted by $\Hom(\T^2)$ and the subset of $\Hom(\T^2)$ consisting in homeomorphisms preserving Lebesgue measure will be denoted by $\Hom(\T^2,\Leb)$. Elements of $\Hom(\T^2)$ will be called \emph{dissipative} and elements of $\Hom(\T^2,\Leb)$ will be called \emph{conservative}. As usual, these two spaces are equipped with the metric of uniform convergence.

We denote by $\Pb$ the set of probability measures on $\T^2$, equipped with a distance $\dist$ compatible with the weak-* topology; by Banach-Alaoglu-Bourbaki theorem $\Pb$ is compact. Let $f$ be a homeomorphism of the torus $\T^2$ homotopic to the identity\footnote{From now every homeomorphism will be supposed homotopic to the identity.}. For $x\in \T^2$, we denote by $p\omega (x)$ the set of limit points of the sequence
\[\left\{\frac{1}{n}\sum_{k=0}^{n-1}\delta_{f^k(x)}\right\}_{n\in\N^*}.\]
It is a compact subset of the set $\mathcal M(f)$ of $f$-invariant Borel probability measures.

For $K\subset \T^2$ we will denote by $\operatorname{diam}_{int}(K)$ the diameter of the biggest euclidean ball included in $K$. By $K'\subset\subset K$ we mean that there exists an open set $O$ such that $K'\subset O\subset K$. In the sequel the set $K$ will be called \emph{strictly periodic} if there exists an integer $i>0$ such that $f^i(K)\subset\subset K$.

\subsection{Generic properties}

The topological spaces $\Hom(\T^2)$ and $\Hom(\T^2,\Leb)$ are Baire spaces (see \cite{MR2931648}), \emph{i.e.} in these spaces the intersection of every countable collection of dense open sets is dense. We call $G_\delta$ a countable intersection of open sets; a property satisfied on at least a $G_\delta$ dense set is called \emph{generic}. Note that in a Baire space generic properties are stable under intersection. Sometimes we will use the phrase ``for a generic homeomorphism $f\in \Hom(\T^2)$ (resp. $\Hom(\T^2,\Leb)$), we have the property $(P)$''. By that we will mean that ``there exists a $G_\delta$ dense subset $G$ of $\Hom(\T^2)$ (resp. $\Hom(\T^2,\Leb)$), such that every $f\in G$ satisfy the property $(P)$''.

\subsection{Rotation sets}

The definition of the rotation set is made to mimic the rotation number for homeomorphisms of the circle. At first sight the natural generalisation to dimension 2 of this notion is the point rotation set, defined as follows. For every homeomorphism $f$ of the torus $\T^2$ homotopic to the identity we take a lift $F : \R^2\to\R^2$ of $f$ to the universal cover $\R^2$ of $\T^2$. The difference with the one dimensional case is that as we lose the existence of a total order on our space, the sequence $\frac{F^n(\tilde x)-\tilde x}{n}$ no longer need to converge. Thus, we have to consider all the possible limits of such sequences, called \emph{rotation vectors}; the set of rotation vectors associated to $\tilde x\in\R^2$ will be denoted by $\overline\rho(\tilde x)$.

Then, the \emph{point rotation set} is defined as $\rho_{pts}(F) = \bigcup_{\tilde x\in\R^2} \overline\rho(\tilde x)$. Unfortunately this definition is not very convenient and it turns out that when we interchange the limits in the previous definition, we obtain the \emph{rotation set}
\[\rho(F) = \bigcap_{M\in\N}\overline{\bigcup_{m\ge M} \left\{\frac{F^m(\tilde x)-\tilde x}{m}\mid \tilde x\in\R^2\right\}}\]
which has much better properties and is easier to manipulate. In particular, it is compact and convex (see \cite{MR1053617}), and it is the convex hull of $\rho_{pts}(F)$. Moreover, it coincides with the \emph{measure rotation set}:  if we denote by $D(F)$ the \emph{displacement function}, defined on $\T^2$ by $D(F)(x) = F(\tilde x)-\tilde x$, where $\tilde x$ is a lift of $x$ to $\R^2$ (one easily checks that this quantity does not depend of the lift), then (recall that $\mathcal M(f)$ is the set of $f$ invariant probability measures)
\[\rho(F) = \left\{\int_{\T^2} D(F)(x)\ \ud\mu\mid \mu\in\mathcal{M}(f)\right\}.\]

Finally, for a homeomorphism $f$ preserving Lebesgue measure, we denote by $\rho_{mean}(F)$ the \emph{mean rotation vector} of $F$:
\[\rho_{mean}(F) = \int_{T^2} D(F)(x)\ \ud\Leb(x).\]

The geometry of the rotation set of a generic dissipative homeomorphism is given by a theorem of A. Passeggi:

\begin{theoreme}[Passeggi, \cite{rata}]\label{ThRata}
On an open and dense set of homeomorphisms $f\in\Hom(\T^2)$, the rotation set is locally constant around $f$ and is equal to a rational polygon. 
\end{theoreme}

We end this paragraph by giving a proof that if $f$ is a generic conservative homeomorphism of the torus then $\rho(F)$ has nonempty interior.

\begin{prop}\label{RotGeneCons}
If $f$ is generic\footnote{In fact on a open dense subset of $\Hom(\T^2,\Leb)$.} among $\Hom(\T^2,\Leb)$, then $\rho(F)$ has nonempty interior.
\end{prop}

\begin{rem}
We do not know the shape of the boundary of the rotation set of a generic conservative homeomorphism. In particular we do not know if it is a polygon or not.
\end{rem}

\begin{proof}[Proof of Proposition \ref{RotGeneCons}]
We use an argument due to S. Crovisier. If $\rho(F)$ consists in a single point, we use classical perturbation techniques for conservative homeomorphisms (see \cite{MR2931648} or \cite{Daal-chao}) to create a persistent periodic point $x$ for $f$. Then by composing by a small rotation of the torus we can move a little the mean rotation vector; in particular as the rotation set still contains the rotation vector of the periodic point $x$, it is not reduced to a single point. Now if the rotation set is a segment, by a $C^0$ ergodic closing lemma we can create a persistent periodic point whose rotation vector is close to the mean rotation vector in the following way. A small perturbation allows us to suppose that the homeomorphism we obtained, still denoted by $f$, is ergodic (see \cite{Oxto-meas}). We then choose a recurrent point $y\in\T^2$ which verifies the conclusion of Birkhoff's theorem: for $N$ large enough, the measure $\frac{1}{N}\sum_{k=0}^{N-1} \delta_{f^k(y)}$ is close to Lebesgue measure. As this 
point is recurrent, by making a little perturbation, we can make it periodic and even persistent; by construction $\overline\rho (y)$ is close to the mean rotation vector. We now have two persistent periodic points, say $x$ and $y$, whose rotation vectors are different. It then suffices to compose by an appropriate rotation such that the mean rotation vector goes outside of the line generated by these two rotation vectors, and to repeat the construction to find a persistent periodic point whose rotation vector is close to this new mean rotation vector.
\end{proof}

\subsection{Observable measures}

From the ergodic point of view, we could be tempted to define the observable rotation set to be the set of rotation vectors associated to physical measures (or SRB measures, see \cite{MR1933431}), which are defined to state which measures can be observed in practice. However, such measures do not need to exist for every dynamical system, in this case the associated observable rotation set would be empty. To solve this problem of non existence of physical measures, E. Catsigeras and H. Enrich have defined in \cite{MR2852870} the weaker notion of observable measure:

\begin{definition}\label{DefObs}
A probability measure $\mu$ is \emph{observable} for $f$ if, for every $\varep>0$, the set 
\begin{equation}\label{defAep}
A_\varep(\mu) = \{x\in \T^2\mid \exists \nu\in p\omega(x) : \dist(\nu,\mu)<\varep \}
\end{equation}
has positive Lebesgue measure. The set of observable measures is denoted by $\Obs(f)$.
\end{definition}

The very interesting property of these measures is that, unlike physical measures, they always exist. More precisely, the set $\Obs(f)$ is a nonempty compact subset of the set of invariant measures of $f$ containing the set of physical measures (see \cite{MR2852870}).

\begin{rem}\label{remConj}
The behaviour of observable measures is compatible with topological conjugacy (whereas for physical measures the conjugacy needs to preserve null sets): if $\mu$ is observable for $f$ and $h$ is a homeomorphism, then $h^*\mu$ is observable for $hfh^{-1}$.
\end{rem}

\begin{ex}\label{ExObsMes}
If $f = \operatorname{Id}$, then $\Obs(f) = \{\delta_x\mid x\in X\}$, but $f$ has no physical measure.
\end{ex}

\begin{prop}\label{PointDissip}
If $f$ is generic among $\Hom(\T^2)$, then
\[\Obs(f) = \operatorname{Cl}\{\delta_\omega\mid \omega\text{ is a Lyapunov stable periodic orbit}\},\]
where $\operatorname{Cl}$ denotes the closure.
\end{prop}

Thus, $f$ has a lot of observable measures\footnote{The set of Lyapunov stable periodic orbits is a Cantor set.}, but no physical measure (see \cite{MR3027586}).

\begin{proof}[Proof of Proposition \ref{PointDissip}]
For the easy inclusion it suffices to remark that every stable measure supported by a Lyapunov stable periodic orbit is observable. For the other inclusion we need the following lemma:

\begin{lemme}\label{LyapGene}
Let $f$ be a generic dissipative homeomorphism of $\T^2$. Then for every strictly periodic topological ball $O$ (\emph{i.e.} there exists $i>0$ such that $f^i(O)\subset\subset O$) there exist a Lyapunov stable periodic point $x\in O$ (\emph{i.e.} for every $\varep>0$ there exists $\delta>0$ such that if $d(x,y)<\delta$ then for every $n\in\N$ we have $d(f^n(x),f^n(y))<\varep$).
\end{lemme}

\begin{proof}[Proof of Lemma \ref{LyapGene}]
We begin by choosing a countable basis of closed sets of the torus: for example we can take $\mathcal K_N$ the set of unions of closed squares whose vertices are multiples of $2^{-N}$. We also denote by $\mathcal B$ the set of all topological balls of $\T^2$. We can now define $E_{k,\varep,N}$ as the set of homeomorphisms such that each large enough strictly periodic ball contains a smaller strictly periodic ball with the same period:
\[E_{k,\varep,N} = \left\{ f\in\Hom(\T^2)\ \middle\vert\ \begin{array}{l}
 \forall K\in\mathcal K_N\cap\mathcal B \text{ s.t. } \exists i\le k \text{ s.t. }\\
 f^i(K) \subset\subset K \text{ and } \operatorname{diam}_{int}(K)>\varep,\\
 \exists K'\subset K, K'\in\mathcal B \text{ s.t.\,} \operatorname{diam}(K')<\varep/2\\
 \text{and } f^i(K') \subset\subset K'
\end{array}\right\}.\]
Then for every $k,\varep,N$, it is straightforward that the set $E_{k,\varep,N}$ is an open subset of $\Hom(X)$. To show that it is dense it suffices to apply Brouwer's theorem to each $K$ such that $f^i(K) \subset\subset K$ and to make the periodic point attractive.

First of all, remark that for every topological ball $K$ with nonempty interior which is strictly $i$-periodic, there exits $N\in\N$ and a smaller topological ball $\tilde K\subset K$ which is strictly $i$ periodic such that $\tilde K \in \mathcal K_N$. It implies that if $f$ belongs to the $G_\delta$ dense set $\bigcap_{k,\varep,N} E_{k,\varep,N}$, then for every topological ball $K$ with nonempty interior which is strictly $i$-periodic, there exits $N\in\N$ and a twice smaller topological ball $K'\subset K$ which is strictly $i$ periodic. Taking the intersection of such balls, we obtain a periodic point with period $i$ which is Lyapunov stable by construction.
\end{proof}

We now return to the proof of Proposition \ref{PointDissip}. Let $f$ be a generic dissipative homeomorphism, $\mu\in \Obs(f)$ and $\varep>0$. By hypothesis $\Leb(A_\varep(\mu))>0$ (see Equation \eqref{defAep}), then $\varep' = \frac12 \min(\varep, \Leb(A_\varep(\mu)))>0$. As $f$ is generic it satisfies the conclusions of the shredding lemma of F. Abdenur and M. Andersson (see \cite{MR3027586}) applied to $f$ and $\varep'$, in particular there exists $B\subset A_\varep(\mu)$ and $O\subset \T^2$ such that:
\begin{itemize}
\item $\Leb (B) > 0$,
\item $O$ is a periodic open set: $\exists i>0 : f^i(O)\subset\subset O$,
\item $\diam (O)<\varep'$,
\item every orbit of every point of $B$ belongs to $O$ eventually.
\end{itemize}
By Lemma \ref{LyapGene}, $O$ contains a Lyapunov stable periodic point whose orbit is denoted by $\omega$; thus for every $x\in B$ and every $\nu\in p\omega(x)$ one has $\dist(\nu,\delta_\omega)<\varep'$. But by hypothesis $\dist(\nu,\mu)<\varep$, then $\dist(\mu, \delta_\omega)<2\varep$, with $\omega$ a Lyapunov stable periodic orbit.
\end{proof}

\begin{lemme}\label{PointCons}
If $f$ is generic among $\Hom(\T^2,\Leb)$, then $\Obs(f) = \{\Leb\}$ coincide with the set of physical measures.
\end{lemme}

\begin{proof}[Proof of Lemma \ref{PointCons}]
A classical theorem of J. Oxtoby and S. Ulam \cite{Oxto-meas} states that a generic conservative homeomorphism $f\in \Hom(\T^2,\Leb)$ is ergodic with respect to Lebesgue measure. But Remark 1.8 of \cite{MR2852870} states that if Lebesgue measure is ergodic, then $\Obs(f) = \{\Leb\}$.
\end{proof}

\section{Observable rotation sets}

\subsection{Definition}

As said before, from the notion of observable measure it is easy to define a notion of observable ergodic rotation set. Another definition, more topologic, seemed reasonable to us for observable rotation sets:

\begin{definition}
\[\rho^{obs}(F) = \Big\{ v\in\R^2 \mid \forall \varep>0,\, \Leb\big\{x\mid \exists u\in \overline\rho (x) : d(u,v)<\varep\big\}>0\Big\}.\]
\[\rho^{obs}_{mes}(F) = \left\{\int_{\T^2} D(F)(x) \ud \mu(x) \mid \mu\in \Obs(f)\right\}.\]
\end{definition}

These two sets are nonempty compact subsets of the classical rotation set, and the first one is even a subset of $\rho_{pts}(F)$. The next lemma states that these two definitions coincide:

\begin{lemme}\label{equal}
$\rho^{obs}_{mes}(F) = \rho^{obs}(F)$.
\end{lemme}

\begin{proof}[Proof of Lemma \ref{equal}]
We first prove that $\rho^{obs}_{mes}(F) \subset \rho^{obs}(F)$. Let $v\in \rho^{obs}_{mes}(F)$ and $\varep>0$. Then there exists $\mu\in \Obs(f)$ such that $v = \int_{\T^2} D(F) \ud \mu$, in particular $\Leb(A_{\varep/2}(\mu))>0$. But if $x\in A_{\varep/2}(\mu)$, then there exists a strictly increasing sequence of integers $(n_i(x))_i$ such that for every $i\ge 0$,
\[\dist \left(\frac{1}{n_i(x)}\sum_{k=0}^{n_i(x)-1} \delta_{f^k(x)},\mu\right)<\varep.\]
Thus,
\[\left|\frac{1}{n_i(x)}\sum_{k=0}^{n_i(x)-1} D(F)(f^k(x)) - \int_{\T^2} D(F) \ud \mu\right|<\varep,\]
in other words the inequality
\[\left|\frac{F^{n_i(x)}(x)-x}{n_i(x)} - v\right|<\varep\]
holds for every $i$ and on a Lebesgue positive measure set of points $x$.
\medskip

For the other inclusion, let $v\in \rho^{obs}(F)$ and set
\[\tilde A_\varep(v) = \{x\in \T^2 \mid \exists u\in\overline\rho(x) : d(u,v)<\varep\}.\]
By hypothesis, $\Leb(\tilde A_\varep(v))>0$ for every $\varep>0$. To each $x\in \tilde A_\varep(v)$ we associate the set $p\omega_\varep^v(x)$ of limit points of the sequence of measures 
\[\frac{1}{n_i(x)}\sum_{k=0}^{n_i(x)-1} \delta_{f^k(x)},\]
where  $(n_i(x))_i$ is a strictly increasing sequence such that
\[\left| \frac{F^{n_i(x)}(x)-x}{n_i(x)}-v\right| <\varep.\]
By compactness of $\Pb$, the set $p\omega_\varep^v(x)$ is nonempty and compact. In the sequel we will use the following easy remark: if $0<\varep<\varep'$ and $x\in \tilde A_\varep$, then $p\omega_\varep^v(x)\subset p\omega_{\varep'}^v(x)$.

By contradiction, suppose that for every $\mu\in\Pb$, there exists $\varep_\mu>0$ such that
\[\Leb\big\{ x\in \tilde A_{\varep_\mu}(v) \mid \exists \nu\in p\omega_{\varep_\mu}^v(x) : \dist(\nu,\mu)<\varep_\mu \big\} =0.\]
By compactness, $\Pb$ is covered by a finite number of balls $B(\mu_j,\varep_{\mu_j})$. Taking $\varep = \min \varep_{\mu_j}$, for every $j$ one has
\[\Leb\big\{ x\in \tilde A_\varep(v) \mid \exists\nu\in p\omega_\varep^v(x) : \dist(\nu,\mu)<\varep_{\mu_j}\big\} =0,\]
thus, as balls $B(\mu_j,\varep_{\mu_j})$ cover $\Pb$,
\[\Leb\big\{ x\in \tilde A_\varep(v) \mid p\omega_\varep^v(x)\cap \Pb\neq\emptyset \big\} =0,\]
which is a contradiction.

Therefore, there exists $\mu_0\in\Pb$ such that for every $\varep>0$,
\[\Leb\big\{ x\in \tilde A_{\varep}(v) \mid \exists\nu\in p\omega_{\varep}^v(x) : \dist(\nu,\mu_0)<\varep \big\} > 0,\]
in particular $\mu_0\in\Obs(f)$. Furthermore, for $\varep>0$, there exists $x\in \tilde A_{\varep}(v)$ and $\mu_x\in p\omega_\varep^v(x)$ such that $\dist(\mu_x,\mu_0)<\varep$. As $\mu_x\in p\omega_\varep^v(x)$, there exists a sequence $(n_i(x))_i$ such that
\[\dist\left(\mu_x\,,\, \frac{1}{n_i(x)}\sum_{k=0}^{n_i(x)-1} \delta_{f^k(x)}\right)<\varep \qquad \text{and} \qquad \left|\frac{F^{n_i(x)}(x)-x}{n_i(x)} - v\right| < \varep.\]
Thus,
\[\dist \left(\mu_0\,,\, \frac{1}{n_i(x)}\sum_{k=0}^{n_i(x)-1} \delta_{f^k(x)}\right)<2\varep.\]
Integrating this estimation according to the function $D(F)$, we obtain:
\[\left|\int_{\T^2} D(F)\ud\mu_0 - \frac{F^{n_i(x)}(x)-x}{n_i(x)}\right| < 2\varep,\]
so
\[\left|\int_{\T^2} D(F)\ud\mu_0 - v\right| < 3\varep,\]
for every $\varep>0$, in other words,
\[v = \int_{\T^2} D(F)\ud\mu_0.\]
\end{proof}

\subsection{Properties of the observable rotation set}

We begin by giving two lemmas which state the dynamical behaviour of the observable rotation sets.

\begin{lemme}\label{iter}
For every $q\in \N$, $\rho^{obs}(F^q) = q\rho^{obs}(F)$.
\end{lemme}

\begin{proof}[Proof of Lemma \ref{iter}]
It suffices to remark that $\overline\rho_{F^q}(x) = q\overline\rho_{F}(x)$ (one inclusion is trivial and the other is easily obtained by Euclidean division).
\end{proof}

\begin{rem}
In general $\rho^{obs}(F^{-1}) \neq -\rho^{obs}(F)$: see for instance the point \ref{F-1} of Example \ref{exrot}.
\end{rem}

\begin{lemme}\label{Dynamic}
If $H$ is a homeomorphism of $\R^2$ commuting with integral translations, then $\rho^{obs}(H\circ F\circ H^{-1}) = \rho^{obs}(F)$.
\end{lemme}

\begin{proof}[Proof of Lemma \ref{Dynamic}]
It follows easily from the fact that the notion of observable measure is stable by conjugacy (see Remark \ref{remConj}).
\end{proof}

We now give a few simple examples of calculation of observable rotation sets.

\begin{ex}\label{exrot}
\begin{enumerate}
\item If $f = \operatorname{Id}$, then $\rho^{obs}(F) = \{(0,0)\}$.
\item If
\[F(x,y) = \left( x+\cos(2\pi x)\, ,\, y\right),\]
then $\rho_{pts}(F) = \rho^{obs}(F) = [-1,1]\times \{0\}$.
\item\label{F-1} If
\[F(x,y) = \left( x+\cos(2\pi y)\, ,\, y+\frac{1}{100}\sin(2\pi y)\right),\]
then $\rho_{pts}(F) = \{(0,-1),(0,1)\}$, but $\rho^{obs}(F) = \{(0,-1)\}$ and $\rho^{obs}(F^{-1}) = \{(0,1)\}$.
\item Let $P$ be a convex polygon with rational vertices. In \cite{MR1176627}, J. Kwapisz
has constructed an axiom A diffeomorphism $f$ of $\T^2$ whose rotation set is the
polygon $P$. It is possible to modify slightly Kwapisz's construction so that all the sinks of
$f$ are fixed points, and so that the union of the basins of these sinks have
full Lebesgue measure. Hence, the observable rotation set of $f_P$ is reduced to
$\{(0,0)\}$.
\end{enumerate}
\end{ex}
\bigskip

We now give the results about the relationship between the rotation set and the observable rotation set in the generic setting. We begin by the dissipative case:

\begin{prop}\label{GeneDissip}
If $f$ is generic among $\Hom(\T^2)$, then $\rho(F) = conv(\rho^{obs}(F))$. If moreover $f$ is generic with a nonempty interior rotation set, then $\rho(F) = \rho^{obs}(F)$.
\end{prop}

\begin{lemme}\label{ExGeneDissip}
If $f$ is generic among $\Hom(\T^2)$ (more precisely, on an open dense subset of $\Hom(\T^2)$), then
\[\rho^{obs}(F) = \operatorname{Cl}\{\overline\rho(\tilde x)\mid x\text{ is a Lyapunov stable periodic point}\}.\]
\end{lemme}

\begin{proof}[Proof of Lemma \ref{ExGeneDissip}]
The inclusion
\[\rho^{obs}(F) \supset \{\overline\rho(\tilde x)\mid x\text{ is a Lyapunov stable periodic point}\}\]
is straightforward. For the other inclusion we need to prove that every observable rotation vector of $F$ can be approximated by rotation vectors of Lyapunov stable periodic points. Let $f$ be a generic homeomorphism (thus satisfying the conclusions of the shredding lemma of F. Abdenur and M. Anderson for every $\delta>0$, see \cite{MR3027586}), $v$ be an observable rotation vector and $\varep>0$. Then $\Leb\{\tilde x\mid d(\overline\rho(\tilde x),v)<\varep\}>0$; in particular by the shredding lemma there exists $x\in\T^2$ such that $d(\overline\rho(\tilde x),v)<\varep$ and that $x$ belongs to a strictly periodic closed set $K$ (\emph{i.e.} there exists $i\in\N^*$ such that $f^i(K)\subset\subset K$) with nonempty interior whose diameter is smaller than $\varep$. By Lemma \ref{LyapGene}, this set $K$ contains a Lyapunov stable periodic point $p$ which $\varep$-shadows the orbit of $x$ eventually. Thus the rotation vectors $v$ and the rotation vector of $p$ are close each other.
\end{proof}

\begin{proof}[Proof of Proposition \ref{GeneDissip}]
Theorem \ref{ThRata} states that for an open dense set of homeomorphisms, the rotation set is a polygon. Then, a theorem of realization of J. Franks \cite[Theorem 3.5]{MR967632} implies that every vertex of this polygon is realized by a periodic point of the homeomorphism, which can be made attractive by a little perturbation of the homeomorphism. Then generically we can find a Lyapunov stable periodic point which shadows the previous periodic point (by Lemma \ref{LyapGene}), in particular it has the same rotation vector. Thus every vertex of $\rho(F)$ belongs to $\rho^{obs}(F)$ and $\rho(F) = conv(\rho^{obs}(F))$.

For the second part of the proposition, another theorem of realization of J. Franks \cite[Theorem 3.2]{MR958891} states that every rational interior point of the polygon is realized as a periodic point of the homeomorphism. For $\varep>0$, one can then find a finite $\varep$-dense subset $R_\varep$ of $\rho(F)$ made of rational points and make the corresponding periodic points of the homeomorphism attractive. Thus, for every $\varep>0$, the set $O_\varep$ made of the homeomorphisms such that every vector of $R_\varep$ is realized by a strictly periodic open subset of $\T^2$ is open and dense in the set of homeomorphisms with nonempty interior rotation set. Applying Lemma \ref{LyapGene} we find a $G_\delta$ dense subset of $O_\varep$ on which every strictly periodic open subset of $\T^2$ contains a Lyapunov stable periodic point; on this set the Hausdorff distance between $\rho(F) = \rho^{obs}(F)$ is smaller than $\varep$. The conclusion of the proposition then easily follows from Baire theorem.
\end{proof}

\begin{rem}
It is not true that $\rho(F) = \rho^{obs}(F)$ holds for a generic homeomorphism: see for instance the point \ref{F-1} of Example \ref{exrot}, where on a neighbourhood of $f$ the set $\rho^{obs}$ is contained in a neighbourhood of the points $(0,-1)$ and $(0,1)$.
\end{rem}

For the conservative case, we recall the result of Proposition \ref{RotGeneCons}: the rotation set of a generic conservative homeomorphism has nonempty interior. The following result states that in this case the observable rotation set is much more smaller, more precisely it consists in a single vector, namely the mean rotation vector.

\begin{prop}\label{ExGeneCons}
If $f$ is generic among $\Hom(\T^2,\Leb)$, then $\rho^{obs}(F) = \{\rho_\Leb(F)\}$, where $\rho_\Leb(F)$ is the mean rotation vector with respect to Lebesgue measure.
\end{prop}

\begin{proof}[Proof of Proposition \ref{ExGeneCons}]
It is easily implied by the fact that Lebesgue measure is the only observable measure (Lemma \ref{PointCons}, which easily follows from Oxtoby-Ulam theorem).
\end{proof}

\section{Discretized rotation sets}\label{DiscRot}

We now define the alternative notion of discretized rotation set. For $N\in\N^*$, we equip the torus $\T^2$ with a \emph{grid of discretization}
\[E_N = \left\{ \left(\frac{i}{2^N},\frac{j}{2^N}\right)\ \middle\vert\ 0\le i,j\le 2^N-1\right\}.\]
We then define the projection $P_N = \T^2\to E_N$ by: for $x\in\T^2$, $P_N(x)$ is (one of) point of $E_N$ which is the nearest to $x$. The \emph{discretization} of a homeomorphism $f : \T^2 \to\T^2$ with respect to the grid $E_N$ is defined as the map $f_N = P_N\circ f : E_N \to E_N$. For each $N$ the map $f_N$ is finite, thus it has a finite number of periodic orbits. To each of these periodic orbits we can associate a rotation vector by applying the same discretization process to the lift $F : \R^2 \to\R^2$ of $f$. The union of these rotation vectors over the periodic points of $f_N$ is denoted by $\rho(F_N)$ and called the \emph{discretized rotation set} on $E_N$. Then the \emph{asymptotic discretized rotation set} is the upper limit of the sets $\rho(F_N)$:
\[\rho^{discr}(F) = \bigcap_{M\in\N}\bigcup_{N\ge M} \rho(F_N).\]

The first result is that for every homeomorphism $f$, the discretized rotation set $\rho(F_N)$ is almost included in the rotation set $\rho(F)$ when $N$ is large enough:

\begin{prop}\label{InclPts}
For every homeomorphism $f$ and every $\varep>0$, it exists $N_0\in\N$ such that for every $N\ge N_0$, one has $\rho(F_N) \subset B(\rho(F),\varep)$, where $B(\rho(F),\varep)$ denotes the set of points whose distance to $\rho(F)$ is smaller than $\varep$. In particular $\rho^{discr}(F)\subset\rho(F)$.
\end{prop}

\begin{proof}[Proof of Proposition \ref{InclPts}]
By definition of the rotation set, for $\varep>0$ there exists $m\in\N$ such that
\[\left\{\frac{F^m(\tilde x)-\tilde x}{m}\mid \tilde x\in\R^2\right\} \subset B(\rho(F),\varep).\]
Then there exists $N_0\in\N$ such that for every $N\ge N_0$, 
\[\left|\frac{F^m(\tilde x)-\tilde x}{m} - \frac{F_N^m(\tilde x_N)-\tilde x_N}{m}\right| \le \varep.\]
This allows us to handle the case of long periodic orbits of the discretizations: by euclidean division, each periodic orbit of $f_N$ of length bigger than $m/\varep$ will be in the $\varep$ neighbourhood of the convex hull of the set
\[\frac{F_N^m(\tilde x_N)-\tilde x_N}{m},\]
so in the $3\varep$-neighbourhood of the rotation set $\rho(f)$.

For the small orbits we argue by contradiction: suppose that there exist $\varep>0$ such that for every $N_0\in\N$ there exists $N\ge N_0$ and $x_N\in E_N$ which is periodic under $f_N$ with period smaller than $m/\varep$ and whose associated rotation vector is not in $B(\rho(F),\varep)$. Then up to take subsequences these periodic points $x_N$ have the same period and converge to a periodic point $x\in\T^2$ whose associated rotation vector (for $F$) is not in $B(\rho(F),\varep)$, which is impossible.
\end{proof}

It remains to study the other inclusion. We begin by the dissipative case.

\begin{prop}\label{RotDiscrDissip}
If $f$ is generic among $\Hom(\T^2)$, then $\rho(F_N)$ tends to $\rho(F)$ for the Hausdorff topology. In particular $\rho^{discr}(F) = \rho(F)$.
\end{prop}

\begin{proof}[Proof of Proposition \ref{RotDiscrDissip}]
The fact that the upper limit of $\rho(F_N)$ is included in $\rho(F)$ follows directly from Lemma \ref{InclPts}.

It remains to prove that the lower limit of $\rho(F_N)$ contains $\rho(F)$. First of all the rotation set is the closure of rotation vectors of Lyapunov stable periodic points (Proposition \ref{PointDissip}). To each one of these points we can associate a periodic closed set $K$ with nonempty interior and with period $\tau$ which has the same rotation vector. Then there exists an open set $O\subset K$ such that for $N$ large enough and $x\in K$ we also have $f_N^\tau(x_N) \in O\subset K$. Thus there exists $i\in\N^*$ such that $f_N^{\tau i}(x_N) = f_N^{2\tau i}(x_N)$ and $f_N^{\tau i}(x_N)$ has the same rotation vector as $K$, thus the same rotation vector as the initial Lyapunov stable periodic point.
\end{proof}

For the conservative case, with the same techniques as in \cite{Guih-discr} we can prove the following result:

\begin{lemme}\label{EnsRotDiscrCons}
If $f$ is generic among $\Hom(\T^2,\Leb)$, then for every finite collection of rotation vectors $\{v_1,\cdots,v_n\}$, each one realized by a periodic orbit, there exists a subsequence $f_{N_i}$ of discretizations such that for every~$i$, $\rho_{N_i}(f) = \{v_1,\cdots,v_n\}$.
\end{lemme}

\begin{proof}[Proof of Lemma \ref{EnsRotDiscrCons}]
We denote by $\mathcal D_q$ the set of subsets of $\Q^2$ made of elements whose coordinates are of the type $p'/q'$, with $0<q'<q$ and $-q^2<p<q^2$. Consider the set
\[\bigcap_{q,N_0}\bigcap_{D\in\mathcal D_q}\bigcup_{N\ge N_0} \left\{\begin{array}{c}
f\in\Hom(\T^2,\Leb) \mid(\forall v\in D,\,v\text{ is realisable by a persistent}\\
\text{periodic point of $f$}) \implies \rho(F_N) = D 
\end{array}\right\}.
\]
To prove the lemma it suffices to prove that this set contains a $G_\delta$ dense. It is obtained in combining the arguments of Proposition 27 and Proposition 39 of \cite{Guih-discr}. We present the main arguments. Given $f$, $N_0$ and $D\in \mathcal D_q$ we take a large $N\ge N_0$ and divide the grid $E_N$ in $2^k$ similar smaller grids such that $2^k \ge \card D$. We want to create a finite map on $E_N$ whose rotation set is $D$, the perturbation of the homeomorphism $f$ is then easily obtained by finite map extension (Proposition 18 of \cite{Guih-discr}). To create such a finite map, on each of the $2^k$ subgrids we apply Lax-Alpern Theorem (see \cite{MR0272983} and \cite{Alpe-newp}, see also Theorem 20 of \cite{Guih-discr}) to obtain a cyclic permutation of the grid $E_N$ which is close to $f$. Then for every $v\in D$ we can choose one of these subgrids and modify this cyclic permutation by following the associated periodic orbit of $f$ in order to obtain a periodic orbit on $E_N$ whose rotation vector is $v$ (see Figure 3 of \cite{Guih-discr}). This creates a finite 
map on $E_N$ whose rotation set is exactly $D$ and which is as close as possible to $f$.
\end{proof}

The combination of the realisation theorem of J. Franks \cite[Theorem 3.2]{MR958891} and the fact that for a generic conservative homeomorphism the rotation set has nonempty interior (Proposition \ref{RotGeneCons}) leads to the corollary:

\begin{coro}\label{CoroRotDiscrCons}
If $f$ is generic among $\Hom(\T^2,\Leb)$, then for every compact subset $K$ of the rotation set of $F$ there exists a subsequence $f_{N_i}$ of discretizations such that $\rho_{N_i}(F)$ tends to $K$ for the Hausdorff topology. In particular $\rho^{discr}(F) = \rho(F)$.
\end{coro}

\section{Numerical simulations}

We have conducted numerical simulations of the rotation sets associated to both conservative and dissipative homeomorphisms. We have made the deliberate choice to choose homeomorphisms whose rotation set is known to be the square $[0,1]^2$. Of course these homeomorphisms are not the best candidates for ``generic'' homeomorphisms, but at least we are sure of what is the shape of the rotation set we want to obtain.

As an example of conservative homeomorphism we have taken $f_1 = Q\circ P$, and for the dissipative homeomorphism we have chosen the very similar expression $f_2 = R\circ Q\circ P$, where
\begin{eqnarray*}
P(x,y) & = & \Big(x\ , \ y+\frac12\big(\cos(2\pi(x+\alpha))+1\big)\\
       &   & + 0.0234\sin^2(4\pi(x+\alpha))\big(\sin(6\pi(x+\alpha))+0.3754\cos(26\pi(x+\alpha))\big)\Big),
\end{eqnarray*}
\begin{eqnarray*}
Q(x,y) & = & \Big(x + \frac12\big(\cos(2\pi (y+\beta))+1\big)\\
			 &   & + 0.0213\sin^2(4\pi (y+\beta))\big(\sin(6\pi(y+\beta))+0.4243\cos(22\pi(y+\beta))\big)\  , \ y\Big),
\end{eqnarray*}
\begin{eqnarray*}
R(x,y) & = & \big(x-0.0127\sin(4\pi (x+\alpha)) + 0.000824\sin(10\pi y)\ ,\\
       &   & y - 0.0176\sin(4\pi (y+\beta)) + 0.000631\sin(12\pi y)\big),
\end{eqnarray*}
with $\alpha = 0.00137$ and $\beta = 0.00159$.

The homeomorphisms $P$ and $Q$ are close to the homeomorphisms
\[\tilde P(x,y) = \Big(x\ , \ y+\frac12\big(\cos(2\pi(x+\alpha))+1\big)\Big)\]
and
\[\tilde Q(x,y) = \Big(x + \frac12\big(\cos(2\pi (y+\beta))+1\big)\Big);\]
it can easily be seen that the rotation set of the homeomorphism $Q\circ P$ is the square $[0,1]^2$, whose vertices are realized by the points $(0,0)$, $(0,1/2)$, $(1/2,0)$ and $(1/2,1/2)$. The perturbations $P$ and $Q$ of $\tilde P$ and $\tilde Q$ are small enough (in $C^2$ topology) to ensure that the rotation set remains the square $[0,1]^2$; these perturbations are made in order to make $f_1$ ``more generic''. The key property of the homeomorphism $R$ is that is has the fixed points of $f_1$ which realize the vertices of $[0,1]^2$ as fixed attractive points; this creates fixed attractive points which realize the vertices of the rotation set. 

We have chosen $R$ to be very close to the identity in $C^1$-topology to ensure that the basins of the sinks and sources are large enough. Indeed, J.-M. Gambaudo and C. Tresser have shown in \cite{Gamb-dif} that, even for dissipative diffeomorphisms defined by very simple formulas, sinks and sources are often undetectable in practice because the size of the their basins are too small.

\captionsetup{width=.47\linewidth}
\begin{figure}[b]
\begin{center}
\makebox[\textwidth]{\parbox{\textwidth}{%
\begin{center}
\begin{minipage}[c]{.49\textwidth}
	\includegraphics[width=6.6cm,trim = 1.2cm 0cm 1.1cm .7cm,clip]{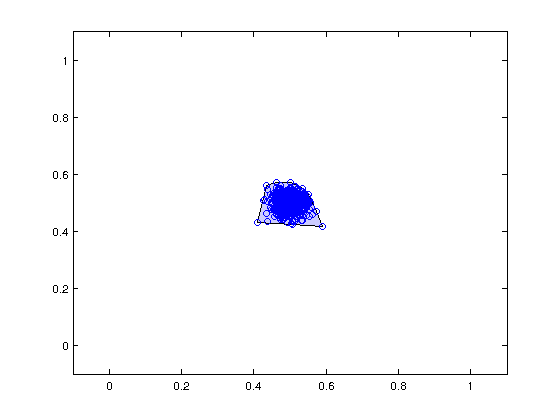}
	\caption{Observable rotation set of $f_1$, $1\,000$ orbits of length $1\,000$ with random starting points, computed with 52 binary digits, $\simeq$ 10s of calculus}\label{Fig1}
\end{minipage}
\begin{minipage}[c]{.49\textwidth}
	\includegraphics[width=6.6cm,trim = 1.2cm 0cm 1.1cm .7cm,clip]{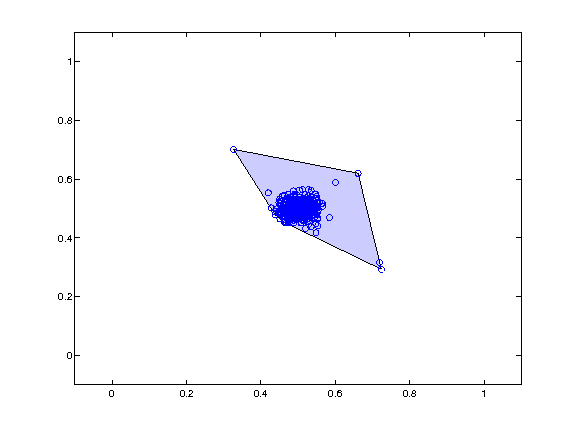}
	\caption{Observable rotation set of $f_1$, $250\,000$ orbits of length $1\,000$ with starting points on the grid $500\times 500$, computed with 52 binary digits, $\simeq$ 45min of calculus}\label{Fig2}
\end{minipage}
\end{center}
\bigskip
\begin{center}
\begin{minipage}[c]{.49\textwidth}
	\includegraphics[width=6.6cm,trim = 1.2cm 0cm 1.1cm .7cm,clip]{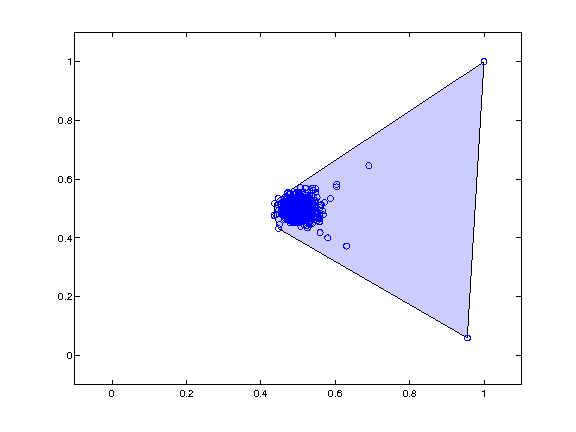}
	\caption{Observable rotation set of $f_1$, $562\,500$ orbits of length $1\,000$ with starting points on the grid $750\times 750$, computed with 52 binary digits, $\simeq$ 1h45min of calculus}\label{Fig3}
\end{minipage}
\begin{minipage}[c]{.49\textwidth}
	\includegraphics[width=6.6cm,trim = 1.2cm 0cm 1.1cm .7cm,clip]{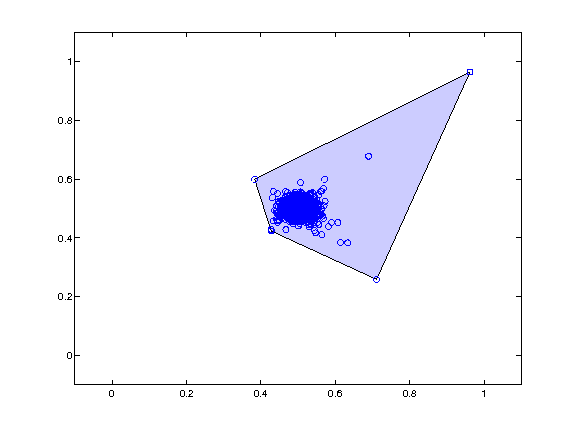}
	\caption{Observable rotation set of $f_1$, $1\,000\,000$ orbits of length $1\,000$ with starting points on the grid $1\,000\times 1\,000$, computed with 52 binary digits, $\simeq$ 3h of calculus}\label{Fig4}
\end{minipage}
\end{center}
}}
\end{center}
\end{figure}

We have made two kinds of simulations of the rotation set.
\begin{itemize}
\item In the first one we have computed the rotation vector of segments of orbits of length $1\,000$ with good precision (52 binary digits), in other words for a starting point $x\in \T^2$ we have computed $\frac{F^{1000}(x)-x}{1000}$. This is maybe the most simple process that can be used to find numerically the rotation set. It should lead to a good approximation of the observable rotation set. In particular, Proposition \ref{GeneDissip} suggests that, for the dissipative homeomorphism $f_2$, we should obtain a set which is close (for Hausdorff distance) to the square $[0,1]^2$, and if not at least a set whose convex hull is this square. On the other hand, for the conservative homeomorphism $f_1$, Proposition \ref{ExGeneCons} suggests that we should only obtain the mean rotation vector, which is close to $(1/2,1/2)$. We have made simulations both with $N$ random starting points and with $N^2$ starting points on grids $N\times N$.
\item In the second kind of simulations we have computed the rotation vectors of the periodic orbits of the discretization $(f_i)_{N}$ on a grid $N\times N$; these simulations calculate the discretized rotation sets. We have also computed the union of the discretized rotation sets for $N_{min}\le N\le N_{max}$, which represents the asymptotic discretized rotation set. The theory tells us that in both conservative and dissipative cases, for some $N$, the discretized rotation set should be close (for Hausdorff distance) to the square $[0,1]^2$; a weaker property would be that its convex hull should be close to this square. Moreover this should also be true for the asymptotic discretized rotation sets.
\end{itemize}

We shall notice that these two methods are formally the same: making simulations on a grid $N\times N$ is equivalent to calculate with $-\log_2 N$ binary digits (for example about 10 for $N=1\,000$). The only difference is that for the second method we use deliberately a very bad numerical precision.

\begin{figure}[t]
\begin{center}
\makebox[\textwidth]{\parbox{\textwidth}{%
\begin{center}
\begin{minipage}[c]{.49\textwidth}
	\includegraphics[width=6.6cm,trim = 1.2cm 0cm 1.1cm .7cm,clip]{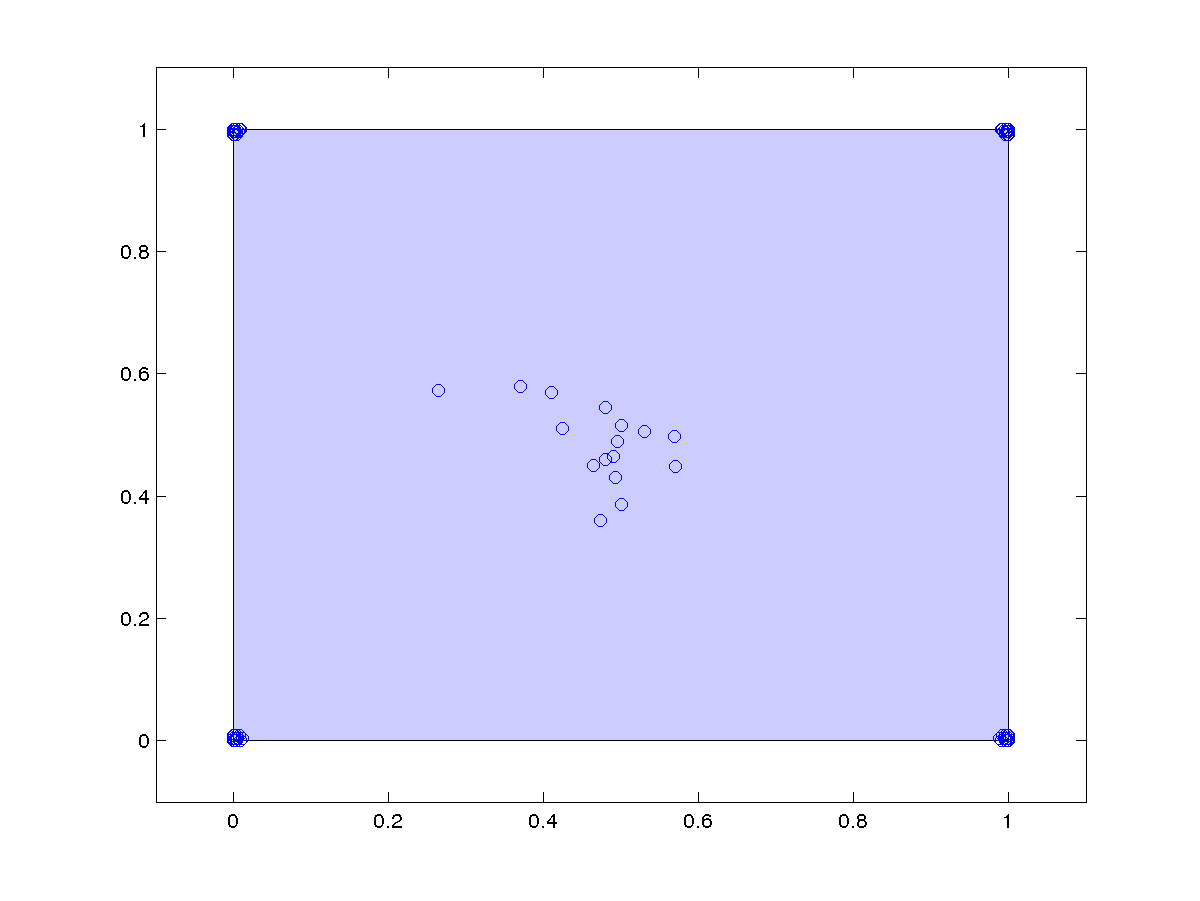}
	\caption{Discretized rotation set of $f_1$ on a grid $100\times 100$, $\simeq$ 0.4s of calculus}\label{Fig5}
\end{minipage}
\begin{minipage}[c]{.49\textwidth}
	\includegraphics[width=6.6cm,trim = 1.2cm 0cm 1.1cm .7cm,clip]{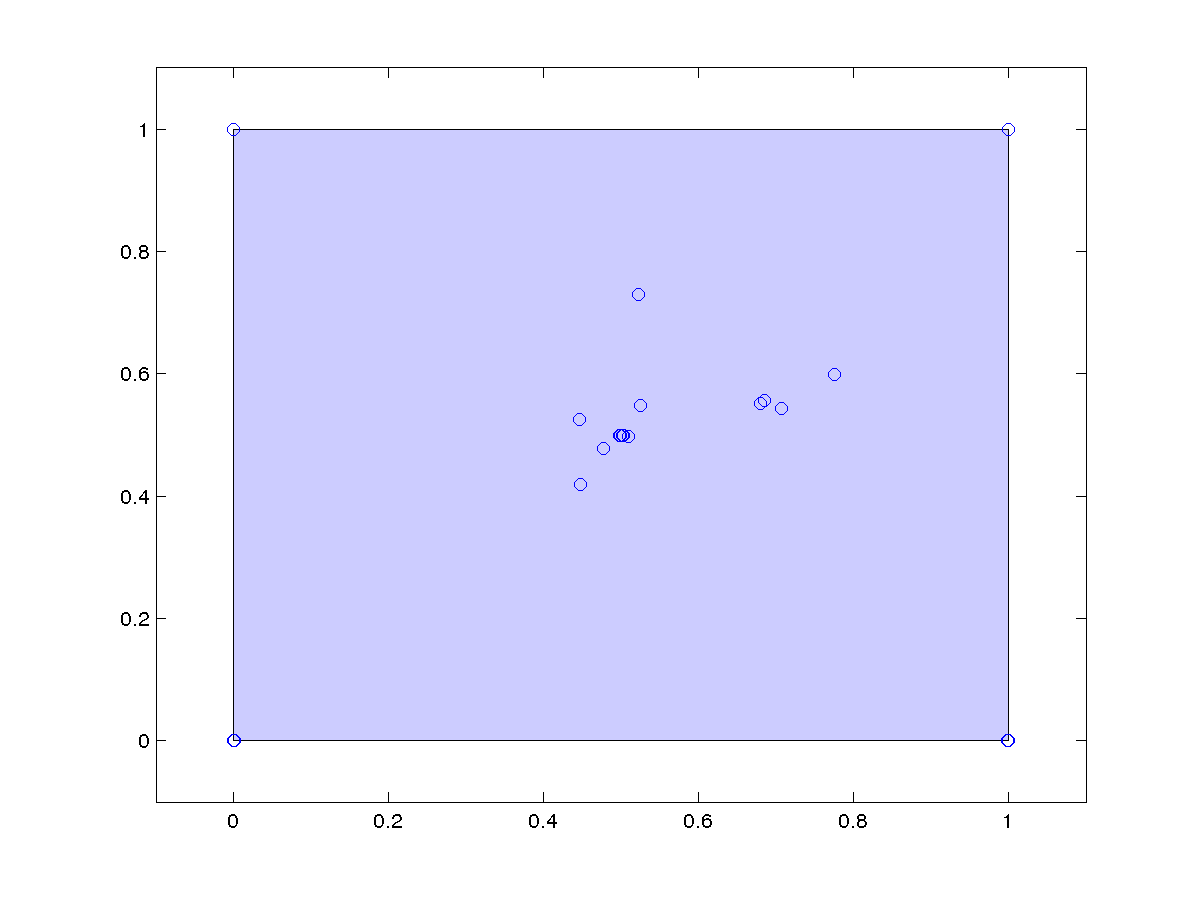}
	\caption{Discretized rotation set of $f_1$ on a grid $1000\times 1000$, $\simeq$ 10s of calculus}\label{Fig6}
\end{minipage}
\end{center}
}}
\end{center}
\end{figure}

\captionsetup{width=9cm}
\begin{figure}[t]
\begin{center}
	\includegraphics[width=9cm,trim = 1.2cm 0cm 1.1cm .7cm,clip]{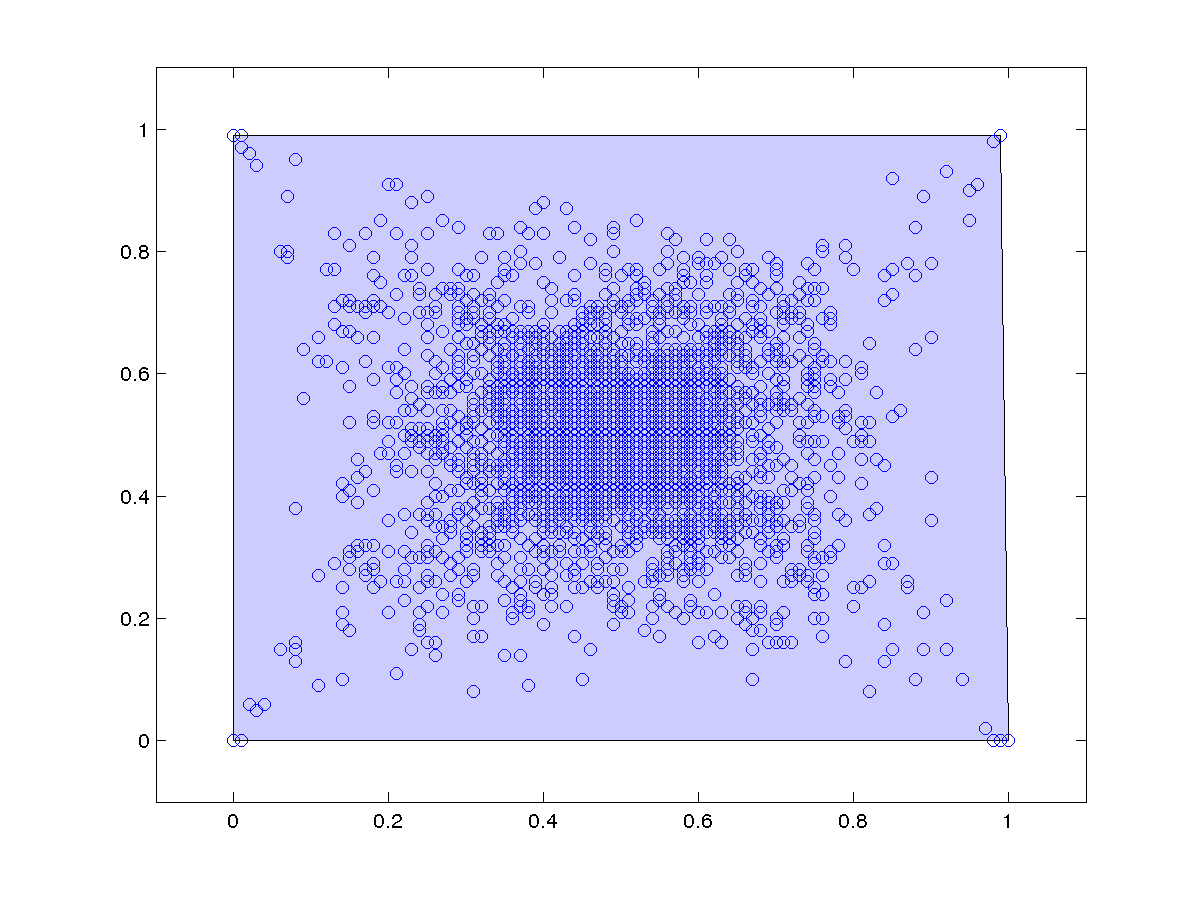}
	\caption{Asymptotic discretized rotation set of $f_1$ as the union of the discretized rotation sets on grids $N\times N$ with $100\le N\le 1\,000$, $\simeq$ 1h30min of calculus}\label{Fig7}
\end{center}
\end{figure}

Moreover, in practice, for a given calculation time, the calculation of the rotation set by discretization (i.e. by the second method) allows to compute much more orbits than the other method. More precisely, the algorithm we have used to compute the asymptotic discretized rotation set visits each point of the grid $N\times N$ once. Thus, for $N^2$ starting points we only have to compute $N^2$ images of the discretization of the homeomorphism on the grid; the number of rotation vectors we obtain is simply the number of periodic orbits of the discretization. So in a certain sense this second algorithm is much more faster than the naive algorithm consisting in computing long segments of orbits. All the simulations have been performed on a computer equipped with a processor Intel Core I5 2.40GHz.
\medskip

In the conservative case, the rotation vectors of the observable rotation set are mainly quite close to the mean rotation vector of $f_1$, as predicted by Proposition \ref{ExGeneCons}. In particular in Figure \ref{Fig1}, all the 1000 rotation vectors of the computed observable rotation set are in the neighbourhood of $(1/2,1/2)$. Thus, the behaviour of these vectors is governed by Birkhoff's ergodic theorem with respect to the ergodic measure $\Leb$; \emph{a priori} this behaviour is quite chaotic and converges slowly: a typical orbit will visit every measurable subset with a frequency proportional to the measure of this set, so the rotation vectors will take time to converge. In Figures \ref{Fig2}, \ref{Fig3} and \ref{Fig4}, we observe a few rotation vectors which are not close to the mean rotation vector; in particular in Figures \ref{Fig3} and \ref{Fig4} we detect the vertex $(1,1)$ of the rotation set (notice that it takes a lot of calculation time to observe this). This phenomenon appears for others sizes of grids, we do not have any explanation for it. We also notice that the shape of the obtained observable rotation set does depend a lot on the size of the grid. Anyway, even with 3 hours of calculus we are unable to recover the initial rotation set of the homeomorphism.

\captionsetup{width=.47\linewidth}
\begin{figure}[b]
\begin{center}
\makebox[\textwidth]{\parbox{\textwidth}{%
\begin{center}
\begin{minipage}[c]{.49\textwidth}
	\includegraphics[width=6.6cm,trim = 1.2cm 0cm 1.1cm .7cm,clip]{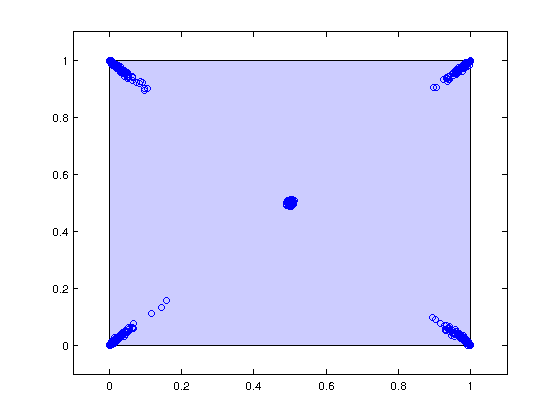}
		\caption{Observable rotation set of $f_2$, orbits with $1\,000$ random starting points, computed with 52 binary digits, $\simeq$ 10s of calculus}\label{Fig8}
\end{minipage}
\begin{minipage}[c]{.49\textwidth}
	\includegraphics[width=6.6cm,trim = 1.2cm 0cm 1.1cm .7cm,clip]{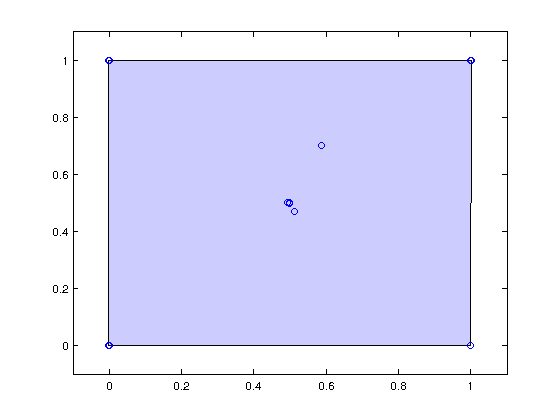}
	\caption{Discretized rotation set of $f_2$ on the grid $1\,000\times 1\,000$, $\simeq$ 10s of calculus\\ \\ \ }\label{Fig9}
\end{minipage}
\end{center}
}}
\end{center}
\end{figure}

On the other hand, the convex hull of the discretized rotation set gives quickly a very good approximation of the rotation set. For example on a grid $100\times 100$ (Figure \ref{Fig5}), with 0.4s of calculus we obtain a rotation set which is already very close to $[0,1]^2$. The same phenomenon occurs for a grid $1\,000\times 1\,000$ (Figure \ref{Fig6}). However, for a single size of grid, we do not obtain exactly the conclusions of Corollary \ref{CoroRotDiscrCons} which states that for some integers $N$ the discretized rotation set should be close to the rotation set for Hausdorff distance; here for each $N$ we only have a few points in the interior of $[0,1]^2$. That is why we represented the union of the discretized rotation sets on grids $N\times N$ with $100\le N\le 1\,000$ (Figure \ref{Fig7}). In this case we recover almost all the rotation set of $f_1$, except from the points which are close to one edge of the square but far from its vertices. The fact that we can obtain very easily the vertices of the rotation set can be due to the fact that in our example $f_1$ these vertices are realized by elliptic fixed points of the homeomorphism (in fact the derivative on this points is the identity). It is possible that if the vertices of the rotation set of a diffeomorphism were realized by hyperbolic periodic points, or elliptic periodic points with bigger period, it would be much more difficult to detect them by looking at the asymptotic discretized rotation set. In short,
\medskip

\begin{center}
\emph{When we calculate with 2 decimal places we find a very good approximation of the rotation set in 0.4s, but when we calculate with 16 decimal places we find a set which does not have much to do with the actual rotation set, even after 3 hours of computation.}
\end{center}
\medskip

In the dissipative case, the observable rotation set is very different from the one in the conservative case, even if the homeomorphism $f_2$ is very close to $f_1$ (approximately $10^{-2}$ close). Indeed, a lot of the obtained rotation vectors are close to one of the vertices of the real rotation set $[0,1]^2$ of $f_2$, the others being located around $(1/2,1/2)$ (see Figure \ref{Fig8}). This is what was predicted by the theory, in particular Lemma \ref{ExGeneDissip}: we detect rotation vectors realized by Lyapunov stable periodic points. The fact that the rotation vectors are not located exactly on the vertices of $[0,1]^2$ can be explained by the slow convergence of the orbits to the attractive points: it may take a while until the orbit become close to one of the Lyapunov stable 
periodic points.

The behaviour of the discretized rotation set for $f_2$ is quite similar to which we observed for $f_1$: the vertices of $[0,1]^2$ are detected and we only have a few points in the interior of the square (see Figure \ref{Fig9}). The small difference with the conservative case is that we have less points in the interior of this square, it can be explained by the fact that a lot of points are attracted by the periodic orbits whose rotation vectors are vertices of $[0,1]^2$ or the centre $(1/2,1/2)$ of the square.

\subsection*{Acknowledgements} I warmly thank François Béguin for his uncountable advices and suggestions about this work. I also thank Sylvain Crovisier for the trick he indicated to me to shorten the proof of Proposition \ref{RotGeneCons}. Finally, the ideas of this paper were born during the workshop ``Surfaces in Sao Paulo'', I would like to thank the organizers for inviting me as well as all the participants with whom I could have had many fruitful discussions during this week in Brazil.

\bibliographystyle{amsalpha}
\bibliography{../../Biblio}

\providecommand{\bysame}{\leavevmode\hbox to3em{\hrulefill}\thinspace}
\providecommand{\MR}{\relax\ifhmode\unskip\space\fi MR }
\providecommand{\MRhref}[2]{%
  \href{http://www.ams.org/mathscinet-getitem?mr=#1}{#2}
}
\providecommand{\href}[2]{#2}
\begin{thebibliography}{Kwa93}

\bibitem[AA13]{MR3027586}
Flavio Abdenur and Martin Andersson, \emph{Ergodic theory of generic continuous
  maps}, Comm. Math. Phys. \textbf{318} (2013), no.~3, 831--855. \MR{3027586}

\bibitem[Alp76]{Alpe-newp}
Steve Alpern, \emph{New proofs that weak mixing is generic}, Invent. Math.
  \textbf{32} (1976), no.~3, 263--278.

\bibitem[CE11]{MR2852870}
Eleonora Catsigeras and Heber Enrich, \emph{S{RB}-like measures for {$C^0$}
  dynamics}, Bull. Pol. Acad. Sci. Math. \textbf{59} (2011), no.~2, 151--164.

\bibitem[DF00]{Daal-chao}
Fons Daalderop and Robbert Fokkink, \emph{Chaotic homeomorphisms are generic},
  Topology Appl. \textbf{102} (2000), no.~3, 297--302.

\bibitem[Fra88]{MR967632}
John Franks, \emph{Recurrence and fixed points of surface homeomorphisms},
  Ergodic Theory Dynam. Systems \textbf{8$^*$} (1988), no.~Charles Conley
  Memorial Issue, 99--107.

\bibitem[Fra89]{MR958891}
\bysame, \emph{Realizing rotation vectors for torus homeomorphisms}, Trans.
  Amer. Math. Soc. \textbf{311} (1989), no.~1, 107--115.

\bibitem[GT83]{Gamb-dif}
Jean-Marc Gambaudo and Charles Tresser, \emph{Some difficulties generated by
  small sinks in the numerical study of dynamical systems: two examples}, Phys.
  Lett. A \textbf{94} (1983), no.~9, 412--414.

\bibitem[Gui12]{MR2931648}
Pierre-Antoine Guih{\'e}neuf, \emph{Propri\'et\'es dynamiques g\'en\'eriques
  des hom\'eomorphismes conservatifs}, Ensaios Matem\'aticos [Mathematical
  Surveys], vol.~22, Sociedade Brasileira de Matem\'atica, Rio de Janeiro,
  2012.

\bibitem[Gui13]{Guih-discr}
\bysame, \emph{Dynamical properties of spatial discretizations of a generic
  homeomorphism}, To appear in Ergodic Theory and Dynamical Systems, 2013.

\bibitem[Her79]{MR538680}
Michael-Robert Herman, \emph{Sur la conjugaison diff\'erentiable des
  diff\'eomorphismes du cercle \`a des rotations}, Inst. Hautes \'Etudes Sci.
  Publ. Math. (1979), no.~49, 5--233. \MR{538680 (81h:58039)}

\bibitem[Kwa92]{MR1176627}
Jaroslaw Kwapisz, \emph{Every convex polygon with rational vertices is a
  rotation set}, Ergodic Theory Dynam. Systems \textbf{12} (1992), no.~2,
  333--339.

\bibitem[Kwa93]{MR1213082}
\bysame, \emph{An estimate of entropy for toroidal chaos}, Ergodic Theory
  Dynam. Systems \textbf{13} (1993), no.~1, 123--129.

\bibitem[Lax71]{MR0272983}
Peter Lax, \emph{Approximation of measure preserving transformations}, Comm.
  Pure Appl. Math. \textbf{24} (1971), 133--135.

\bibitem[LM91]{MR1101087}
Jaume Llibre and Robert~S. MacKay, \emph{Rotation vectors and entropy for
  homeomorphisms of the torus isotopic to the identity}, Ergodic Theory Dynam.
  Systems \textbf{11} (1991), no.~1, 115--128.

\bibitem[LP11]{MR2776399}
Stefano Luzzatto and Pawe{\l} Pilarczyk, \emph{Finite resolution dynamics},
  Found. Comput. Math. \textbf{11} (2011), no.~2, 211--239. \MR{2776399
  (2012b:37210)}

\bibitem[MZ89]{MR1053617}
Micha{\l} Misiurewicz and Krystyna Ziemian, \emph{Rotation sets for maps of
  tori}, J. London Math. Soc. (2) \textbf{40} (1989), no.~3, 490--506.
  \MR{1053617 (91f:58052)}

\bibitem[OU41]{Oxto-meas}
John Oxtoby and Stanislaw Ulam, \emph{Measure-preserving homeomorphisms and
  metrical transitivity}, Ann. of Math. \textbf{42} (1941), no.~2, 874--920.

\bibitem[Pas12]{rata}
Alejandro Passeggi, \emph{Rational polygons as rotation sets of generic
  homeomorphisms of the two-torus}, 2012.

\bibitem[Poi85]{Poincare}
Henri Poincar\'e, \emph{M\'emoire sur les courbes d\'efinies par une \'equation
  diff\'erentielle}, J. de Math Pures Appl (1885), no.~1, 167--244.

\bibitem[You02]{MR1933431}
Lai-Sang Young, \emph{What are {SRB} measures, and which dynamical systems have
  them?}, J. Statist. Phys. \textbf{108} (2002), no.~5-6, 733--754, Dedicated
  to David Ruelle and Yasha Sinai on the occasion of their 65th birthdays.

\end{thebibliography}

\end{document}